\documentclass[a4paper,10pt]{amsart}

\usepackage[british]{babel}
\usepackage[utf8]{inputenc}
\usepackage[T1]{fontenc}
\usepackage{lmodern}

\usepackage{amssymb}
\usepackage{amsmath}
\usepackage{amsthm}
\usepackage{bbm}

\usepackage[shortlabels]{enumitem}
\usepackage{tikz}
\usepackage{marginnote}
\usepackage{mathtools}
\usepackage{orcidlink}

\usepackage{hyperref}

\newcommand{\bbN}{\mathbb{N}}

\newcommand{\bbR}{\mathbb{R}}

\newcommand{\bbZ}{\mathbb{Z}}

\newcommand{\calO}{\mathcal{O}}

\newcommand{\N}{\bbN}
\newcommand{\R}{\bbR}

\newcommand{\Z}{\bbZ}

\DeclarePairedDelimiter{\norm}{\lVert}{\rVert}
\DeclarePairedDelimiter{\abs}{\lvert}{\rvert}
\DeclarePairedDelimiter{\dual}{\langle}{\rangle}
\DeclarePairedDelimiter{\set}{\lbrace}{\rbrace}

\DeclareMathOperator{\nspan}{span}

\DeclareMathOperator{\dom}{dom}

\newcommand{\ud}{\mathrm{d}}

\newcommand{\ue}{\mathrm{e}}

\theoremstyle{definition}
\newtheorem{definition}{Definition}[section]
\newtheorem{remark}[definition]{Remark}

\theoremstyle{plain}
\newtheorem{proposition}[definition]{Proposition}
\newtheorem{lemma}[definition]{Lemma}
\newtheorem{theorem}[definition]{Theorem}

\numberwithin{equation}{section}

\begin{document}

\title[Individual vs.\ Uniform Eventual Positivity]{Individual and Uniform Eventual Positivity of Self-Adjoint Semigroups}

\author{Alexander Dobrick \orcidlink{0000-0002-3308-3581}}
\address[A.~Dobrick]{Alexander Dobrick, Christian-Albrechts-Universität~zu~Kiel (Alumni), Arbeitsbereich~Analysis, 24118 Kiel, Germany}
\email{alexander.dobrick.math@web.de}

\date{\today}

\begin{abstract}
We construct a real self-adjoint semigroup on an $\ell^2$-space which is individually eventually strongly positive with respect to a strictly positive vector, but whose operators fail to be positive at arbitrarily large times. Moreover, the generator of this semigroup has compact resolvent. Consequently, individual and uniform eventual positivity are not equivalent for self-adjoint semigroups on $L^2$-spaces, even under the additional assumption that the generator has compact resolvent. The construction consists of a positive self-adjoint rank-one perturbation of a diagonal operator and a sequence of small spectral shifts on antisymmetric two-point modes.
\end{abstract}

\maketitle

\section{Introduction}

Positivity is one of the central structural assumptions in the spectral and asymptotic theory of operator semigroups. In the classical setting, one requires each operator of a semigroup to leave the positive cone invariant. This assumption is natural for many second-order evolution equations and gives access to the methods of infinite-dimensional Perron--Frobenius theory. For various equations of higher order, however, positivity fails for small times even though positive behaviour emerges asymptotically. This observation has led to the systematic study of \emph{eventually positive semigroups}. The foundations of the theory were developed by several authors (see \cite{Daners2016a, Daners2016b}), and in recent years the theory has reached considerable maturity (cf.~\cite{Arora2022a, Arora2023a, Arora2025a, Arora2021a, Arora2022b, Arora2022c, Arora2023b, Arora2023c, Arora2024a, Arora2025b, Daners2017, Daners2018a, Daners2018b, Daners2023, GlueckDISS, Glueck2022, MuiDISS, Mui2023}). Moreover, the lecture notes \cite{Arora2026a}, written for the 29th Internet Seminar, provide a comprehensive account of the theory.

There are two natural forms of eventual positivity. Individual eventual positivity allows the time after which an orbit becomes positive to depend on its initial value, whereas uniform eventual positivity requires one common time for all positive initial values. On finite-dimensional Banach lattices these notions coincide, since positivity can be tested on finitely many positive basis vectors. This argument breaks down in infinite dimensions, and the distinction between individual and uniform eventual positivity becomes substantial.

There is, however, a positive result under an appropriate smoothing assumption. Let $(\Omega, \Sigma, \mu)$ be a $\sigma$-finite measure space, let $0\leq u\in L^2(\Omega)$ and let $A$ be the generator of a self-adjoint semigroup on $L^2(\Omega)$. It follows from \cite[Corollary~13.2.2]{Arora2026a} that individual eventual strong positivity with respect to $u$ is equivalent to uniform eventual positivity with respect to $u\otimes u$ if
\begin{align*}
    \dom(A^m)\subseteq L^2(\Omega)_u
\end{align*}
for some $m\in\N$, where $L^2(\Omega)_u$ denotes the principal ideal generated by $u$. In fact, the proof only requires that $u$ be strictly positive almost everywhere and that $\ue^{t_0A}L^2(\Omega,\mu)\subseteq L^2(\Omega)_u$ for some $t_0>0$. Thus, sufficient regularisation into the principal ideal forces the individual and uniform notions to coincide.

Without such a smoothing assumption, individual and uniform eventual positivity are not equivalent in general (see \cite[Examples~11.1.2~and~13.1.2]{Arora2026a}). The Banach lattices in those counterexamples are, however, far from being $L^2$-spaces. Moreover, the semigroups exhibit considerable symmetry, which suggests that related constructions might lead to self-adjoint counterexamples on Hilbert lattices. This observation motivates the question considered here.

Self-adjoint semigroups on Hilbert spaces form a particularly rigid class: the spectral theorem provides a precise description of each orbit, while compactness of the resolvent reduces the asymptotic behaviour to isolated eigenspaces of finite multiplicity. Open Problem~14.1.1(a) in \cite{Arora2026a} asks whether individual and uniform eventual positivity are equivalent for real self-adjoint semigroups on $L^2$-spaces. Moreover, part~(b) asks whether individual eventual strong positivity with respect to a strictly positive vector $u$ is equivalent to uniform eventual positivity with respect to $u\otimes u$. The problem remains open there even under the additional assumption that the generator has compact resolvent. The present article answers both questions in the negative under precisely this additional assumption.

\subsection*{Contributions of this article}

The main result of the present article is the following theorem, which provides a natural counterexample to the open question.

\begin{theorem} \label{theorem:main}
    There exist a real self-adjoint $C_0$-semigroup $(\ue^{tA})_{t \geq 0}$ on a Hilbert lattice $H$ and a strictly positive vector $u \in H$ such that the following assertions hold:
    \begin{enumerate}[\upshape (i)]
        \item The generator $A$ has compact resolvent.
        \item The semigroup $(\ue^{tA})_{t \geq 0}$ is individually eventually strongly positive with respect to $u$.
        \item The operator $\ue^{nA}$ is not positive for each sufficiently large $n \in \N$.
    \end{enumerate}
    In particular, $(\ue^{tA})_{t \geq 0}$ is neither uniformly eventually positive nor uniformly eventually positive with respect to $u \otimes u$.
\end{theorem}

The construction separates the symmetric and antisymmetric modes of $H \coloneqq \ell^2(\Z)$. On the symmetric part, we consider a positive self-adjoint rank-one perturbation of a diagonal operator. The resulting reference semigroup is positive and converges in operator norm to the rank-one projection $u\otimes u$ (see Proposition~\ref{proposition:reference-semigroup}). We then shift the eigenvalues on the antisymmetric two-point modes by a sequence which is small compared with the corresponding coordinates of $u$ (see Proposition~\ref{proposition:semigroup-representation}). Consequently, each fixed positive orbit remains eventually strongly positive with respect to $u$ (see Proposition~\ref{proposition:individual}). At suitably chosen times, however, the shifted antisymmetric mode dominates the corresponding symmetric correction and produces a negative off-diagonal matrix entry (see Proposition~\ref{proposition:failure-of-uniform-eventual-positivity}). The construction therefore distinguishes individual from uniform eventual positivity in this specific Hilbert space setting.

\subsection*{Organization of the article}

The paper is organised as follows. In Section~\ref{section:construction-of-the-semigroup} we construct the reference generator $A_0$, introduce the compact perturbation on the antisymmetric modes and derive a representation formula for the semigroup generated by $A$. Section~\ref{section:eventual-positivity-properties} contains the proof of individual eventual strong positivity and the estimate which yields negative matrix entries at arbitrarily large times. Finally, Appendix~\ref{section:scalar-convolution-estimate} provides the scalar convolution result used to control the symmetric contribution.

\subsection*{Basic notions and notation}

We collect some of the specific terminology used throughout this article. For a comprehensive treatment of Banach lattices and related notions, we refer to the classical monographs \cite{Aliprantis1999, Aliprantis2006, MeyerNieberg1991, Schaefer1974, Zaanen2012}. For basic results from the theory of strongly continuous semigroups we refer to \cite{Batkai2017, Bobrowski2016, Davies1980, Engel2000, Engel2006, Goldstein1985, Hille1957, Lunardi1995, NagelEd, Pazy1983, Tanabe1979}.

We use the following terminology throughout. Let $H$ be a Hilbert lattice and let $(\ue^{tA})_{t\geq0}$ be a real $C_0$-semigroup on $H$. The semigroup is called \emph{individually eventually positive} if, for each $0\leq f\in H$, there exists $t_f\geq0$ such that $\ue^{tA}f\geq0$ for all $t\geq t_f$. It is called \emph{uniformly eventually positive} if there exists $t_0\geq0$ such that $\ue^{tA}\geq0$ for all $t\geq t_0$.

For $0<u\in H$, we write $f\succeq u$ if $f\geq cu$ for some $c>0$. The semigroup $(\ue^{tA})_{t\geq0}$ is called \emph{individually eventually strongly positive with respect to $u$} if, for each $0<f\in H$, there exists $t_f\geq0$ such that $\ue^{tA}f\succeq u$ for all $t\geq t_f$. Moreover, it is called \emph{uniformly eventually positive with respect to $u\otimes u$} if there exist $t_0\geq0$ and $c>0$ such that $\ue^{tA}\geq c(u\otimes u)$ for all $t\geq t_0$.

For $v,w\in H$, the rank-one operator $v\otimes w$ is given by $(v\otimes w)x\coloneqq v\dual{w,x}$ for all $x\in H$.

\section{Construction of the semigroup} \label{section:construction-of-the-semigroup}

Consider the Hilbert space $H \coloneqq \ell^2(\Z)$ and denote the canonical orthonormal basis of $\ell^2(\Z)$ by $(e_k)_{k \in \Z}$. We begin with the decomposition of $H$ into its even and odd parts.

\begin{lemma}
    The space $H$ decomposes as $H = H_{\mathrm{even}} \oplus H_{\mathrm{odd}}$, where
    \begin{align*}
        H_{\mathrm{even}}
        &\coloneqq
        \set{x\in\ell^2(\Z) : x_n=x_{-n}\text{ for all }n\in\Z}, \\
        H_{\mathrm{odd}}
        &\coloneqq
        \set{x\in\ell^2(\Z) : x_n=-x_{-n}\text{ for all }n\in\Z}.
    \end{align*}
    Moreover, the following assertions hold:
    \begin{enumerate}[\upshape (i)]
        \item The sequence $(p_n)_{n \in \N_0}$ is an orthonormal basis of $H_{\mathrm{even}}$, where $p_0 \coloneqq e_0$ and
        \begin{align*}
            p_n \coloneqq \frac{1}{\sqrt{2}}(e_n+e_{-n}) \in \ell^2(\Z), \quad n \in \N.
        \end{align*}
        \item The sequence $(q_n)_{n \in \N}$ is an orthonormal basis of $H_{\mathrm{odd}}$, where
        \begin{align*}
            q_n \coloneqq \frac{1}{\sqrt{2}} (e_n - e_{-n}) \in \ell^2(\Z), \quad n \in \N.
        \end{align*}
    \end{enumerate}
\end{lemma}

Next, we consider the diagonal operator $D \colon \dom(D) \to \ell^2(\Z)$ with
\begin{align*}
    \dom(D) \coloneqq \set{x \in \ell^2(\Z) : \sum_{k \in \Z \setminus \{0\}} k^2 \abs{x_k}^2 < \infty}
\end{align*}
given by
\begin{align*}
    D e_k =
    \begin{cases}
        \abs{k} e_k, \quad &\text{if } k \neq 0, \\
        e_0, \quad &\text{if } k=0,
    \end{cases}
\end{align*}
and collect some of its properties in the following lemma.

\begin{lemma} \label{lemma:properties-of-d}
    Let $D \colon \dom D \to \ell^2(\Z)$ be the diagonal operator as defined above. Then the following assertions hold:
    \begin{enumerate}[\upshape (i)]
        \item The operator $D$ is self-adjoint and is positive in the Hilbert space sense.
        \item $D$ has compact resolvent.
        \item Both $H_{\mathrm{even}}$ and $H_{\mathrm{odd}}$ reduce $D$, and one has $Dq_n = nq_n$ for all $n \in \N$.
    \end{enumerate}
\end{lemma}

\begin{proof}
    (i): This is obvious.

    (ii): Clearly, one has $\sigma_p(D) = \N$ and each eigenvalue of $D$ has finite multiplicity; more precisely, one has $\ker(I-D)=\nspan\set{e_{-1},e_0,e_1}$ and $\ker(nI-D)=\nspan\set{e_{-n},e_n}$ for all $n\geq2$. Moreover, $R(-1,D)$ is diagonal with respect to $(e_k)_{k\in\Z}$, and its diagonal entries converge to $0$ as $\abs{k}\to\infty$. Hence, $R(-1,D)$ is an operator-norm limit of finite-rank operators and therefore compact.

    (iii): This follows from a straightforward calculation.
\end{proof}

Now consider the sequence $u = (u_k)_{k \in \Z} \in \ell^2(\Z)$ given by
\begin{align*}
    u_k \coloneqq
    \begin{cases}
        \ue^{-3 k^2/2}, &\quad \text{if } k \neq 0,  \\
        \big(1 - 2 \sum_{k = 1}^\infty \ue^{-3 k^2}\big)^{1/2}, &\quad \text{if } k=0.
    \end{cases}
\end{align*}
Since
\begin{align*}
    2\sum_{k=1}^\infty \ue^{-3k^2}
    \leq 2\sum_{k=1}^\infty \ue^{-3k}
    =\frac{2\ue^{-3}}{1-\ue^{-3}}<1,
\end{align*}
one has $u_0>0$, and hence $u$ is strictly positive. Moreover, $\norm{u}_2=1$ and $u\in\dom D$. We define the operator
\begin{align*}
    A_0 \coloneqq -D + \frac{D u\otimes D u}{\dual{Du, u}}, \qquad \dom(A_0)\coloneqq\dom(D).
\end{align*}
The following proposition collects some properties of this operator.

\begin{proposition} \label{proposition:properties-of-a0}
    Let $A_0$ be the operator on $\ell^2(\Z)$ as defined above. Then the following assertions hold:
    \begin{enumerate}[\upshape (i)]
        \item The operator $A_0$ is real and self-adjoint.
        \item One has $A_0 \leq 0$ in the Hilbert space sense, and $s(A_0) = 0$ is a simple eigenvalue of $A_0$.
        \item $A_0$ has compact resolvent.
        \item Both $H_{\mathrm{even}}$ and $H_{\mathrm{odd}}$ reduce $A_0$ and one has $A_0|_{H_{\mathrm{odd}}} = -D|_{H_{\mathrm{odd}}}$.
    \end{enumerate}
\end{proposition}

\begin{proof}
    (i): $A_0$ is clearly real. The rank-one operator
    \begin{align*}
        K \coloneqq \frac{Du \otimes Du}{\dual{Du,u}}
    \end{align*}
    is bounded, real and self-adjoint. Since $-D$ is self-adjoint by Lemma~\ref{lemma:properties-of-d}(i), it follows from the bounded perturbation theorem for self-adjoint operators that $A_0$ is self-adjoint on $\dom A_0 = \dom D$.

    (ii): By definition of $A_0$, one has
    \begin{align*} A_0u = -Du + \frac{Du}{\dual{Du,u}}\dual{Du,u} = 0.
    \end{align*}
    Thus, $u \in \ker A_0$. Let $x \in \dom A_0$. Since
    \begin{align*}
        \dual{Du, u} = \norm{D^{1/2} u}^2,
    \end{align*}
    the Cauchy--Schwarz inequality yields
    \begin{align*}
        \dual{A_0 x, x} &= - \dual{D x, x} + \frac{\abs{\dual{D u, x}}^2}{\dual{D u, u}} \\
        &= -\norm{D^{1/2} x}^2 + \frac{\abs{\dual{D^{1/2} u, D^{1/2} x}}^2}{\norm{D^{1/2} u}^2} \leq 0.
    \end{align*}
    Hence, $A_0 \leq 0$. Moreover, equality holds if and only if equality holds in the Cauchy--Schwarz inequality, which is equivalent to $D^{1/2}x$ being a scalar multiple of $D^{1/2}u$. Since $D$ is injective, this is equivalent to $x$ being a scalar multiple of $u$. Consequently, $\ker A_0$ is spanned by $u$. In particular, $s(A_0) = 0$ is a simple eigenvalue of $A_0$.

    (iii): Since $K$ is bounded, $\norm{K R(\lambda, -D)} < 1$ for sufficiently large $\lambda > 0$. Then
    \begin{align*}
        R(\lambda, A_0) = R(\lambda, -D)
        (I-K R(\lambda, -D))^{-1}.
    \end{align*}
    Since $D$ has compact resolvent, $R(\lambda, A_0)$ is compact. Consequently, $A_0$ has compact resolvent as well.

    (iv): This follows from a straightforward calculation.
\end{proof}

We next show that $A_0$ generates a positive self-adjoint semigroup on $\ell^2(\Z)$.

\begin{proposition} \label{proposition:reference-semigroup}
The operator $A_0$ generates a self-adjoint $C_0$-semigroup $(\ue^{tA_0})_{t \geq 0}$ on $\ell^2(\Z)$ which is positive in the Banach lattice sense. Moreover,
\begin{align*}
    \ue^{tA_0} \to u \otimes u
\end{align*}
with respect to the operator norm as $t \to \infty$.
\end{proposition}

\begin{proof}
    Since $A_0 \leq 0$ by Proposition~\ref{proposition:properties-of-a0}(ii), the proposition in \cite[p.~91]{Engel2000} implies that $A_0$ generates a self-adjoint $C_0$-semigroup $(\ue^{tA_0})_{t \geq 0}$ on $\ell^2(\Z)$.

    Moreover, the semigroup generated by $-D$ is positive in the Banach lattice sense since each operator $\ue^{-tD}$ is diagonal with positive diagonal entries. Furthermore, $Du > 0$ and thus
    \begin{align*}
        \frac{Du \otimes Du}{\dual{Du,u}} \geq 0
    \end{align*}
    in the sense of Banach lattices. Hence, the bounded positive perturbation theorem (cf.~\cite[Corollary~11.7]{Batkai2017}) implies that $(\ue^{tA_0})_{t \geq 0}$ is positive in the Banach lattice sense.

    By Proposition~\ref{proposition:properties-of-a0}(iii), $A_0$ has compact resolvent. Hence, the spectrum of $A_0$ consists only of isolated eigenvalues of finite multiplicity and has no finite accumulation point. Since $0$ is a simple eigenvalue and $A_0 \leq 0$ by Proposition~\ref{proposition:properties-of-a0}(ii), there exists $\omega > 0$ such that
    \begin{align*}
        \sigma(A_0) \setminus \{0\} \subseteq (-\infty,-\omega].
    \end{align*}
    Since $A_0$ is self-adjoint and $\norm{u} = 1$, the spectral projection associated with the eigenvalue $0$ is given by $P = u \otimes u$. The spectral theorem and the spectral gap obtained above imply the norm convergence as $t \to \infty$.
\end{proof}

We now perturb the odd part of the semigroup. The perturbation will preserve individual eventual strong positivity but destroy uniform eventual positivity.

Define the operator
\begin{align*}
    B \coloneqq \sum_{n=1}^{\infty} \delta_n \, q_n \otimes q_n,
\end{align*}
where $\delta_n \coloneqq n \ue^{-2n^2}$ for $n \in \N$. Since $(q_n)_{n \in \N}$ is an orthonormal sequence and $\delta_n \to 0$ as $n \to \infty$, the series converges in operator norm. Hence, $B$ is compact as an operator-norm limit of finite-rank operators. Moreover, $B$ is real and self-adjoint. Consider the perturbed generator
\begin{align*}
    A\coloneqq A_0 + B, \qquad \dom(A)\coloneqq\dom(A_0).
\end{align*}
The following proposition collects the properties of $A$.

\begin{proposition} \label{proposition:properties-of-a}
    Let $A$ be the operator on $\ell^2(\Z)$ as defined above. Then the following assertions hold:
    \begin{enumerate}[\upshape (i)]
        \item The operator $A$ is real and self-adjoint.
        \item $A$ has compact resolvent.
        \item Both $H_{\mathrm{even}}$ and $H_{\mathrm{odd}}$ reduce $A$ and one has $A|_{H_{\mathrm{even}}} = A_0|_{H_{\mathrm{even}}}$ as well as $Aq_n = -(n - \delta_n)q_n$ for all $n \in \N$.
    \end{enumerate}
\end{proposition}

\begin{proof}
    (i): Since $A_0$ and $B$ are real, so is $A=A_0+B$. Moreover, the bounded perturbation theorem for self-adjoint operators (see \cite[Theorem~V.4.3]{Kato1995a}) implies that $A$ is self-adjoint on $\dom A=\dom A_0$.

    (ii): $A_0$ has compact resolvent by Proposition~\ref{proposition:properties-of-a0}(iii) and $B$ is bounded. Thus, the operator $A$ also has compact resolvent.

    (iii): Note that Lemma~\ref{lemma:properties-of-d}(iii) implies that $A_0 q_n = -n q_n$ for each $n \in \N$. Moreover, $B$ vanishes on $H_{\mathrm{even}}$ and satisfies $B q_n = \delta_n q_n$. Hence, both $H_{\mathrm{even}}$ and $H_{\mathrm{odd}}$ reduce $A$, and one has
    \begin{align*}
        A|_{H_{\mathrm{even}}} = A_0|_{H_{\mathrm{even}}},
    \end{align*}
    whereas $A q_n = -(n - \delta_n) q_n$ for all $n \in \N$.
\end{proof}

The operator $A$ is the generator of the semigroup that will serve as the desired counterexample. We obtain a representation formula for this semigroup.

\begin{proposition} \label{proposition:semigroup-representation}
    The operator $A$ generates a self-adjoint $C_0$-semigroup $(\ue^{tA})_{t \geq 0}$ on $\ell^2(\Z)$. Moreover, one has
    \begin{align*}
        \ue^{tA} = \ue^{tA_0} + \sum_{n=1}^{\infty} h_n(t)\,q_n\otimes q_n,
    \end{align*}
    where $h_n(t) \coloneqq \ue^{-(n-\delta_n)t}-\ue^{-nt}$ for all $t\geq0$ and $n\in\N$.
\end{proposition}

\begin{proof}
    Since $A|_{H_{\mathrm{even}}}=A_0|_{H_{\mathrm{even}}}\leq0$ and $Aq_n=-(n-\delta_n)q_n$ with $0<\delta_n<n$ for each $n\in\N$, one has $A\leq0$. It therefore follows from the proposition in \cite[p.~91]{Engel2000} and Proposition~\ref{proposition:properties-of-a} that $A$ generates a self-adjoint $C_0$-semigroup $(\ue^{tA})_{t \geq 0}$ and $\ue^{tA}q_n = \ue^{-(n - \delta_n) t} q_n$ for each $n\in\N$, whereas $\ue^{tA}|_{H_{\mathrm{even}}} = \ue^{tA_0}|_{H_{\mathrm{even}}}$ for all $t \geq 0$. On the other hand, one has $\ue^{tA_0} q_n = \ue^{-nt}q_n$. Consequently,
    \begin{align*}
        \ue^{tA} - \ue^{tA_0} = \sum_{n=1}^{\infty}(\ue^{-(n - \delta_n) t} - \ue^{-n t}) q_n\otimes q_n = \sum_{n = 1}^\infty h_n(t)\,q_n\otimes q_n
    \end{align*}
    for all $t \geq 0$. Finally, for each fixed $t\geq0$ one has $h_n(t)\to 0$ as $n \to \infty$, and therefore the series converges in operator norm.
\end{proof}

\section{Eventual positivity properties} \label{section:eventual-positivity-properties}

We first establish individual eventual strong positivity for the reference semigroup. We then show that this property persists under the perturbation, while uniform eventual positivity fails.

\begin{proposition} \label{proposition:base-individually-eventually-positive}
    The semigroup $(\ue^{tA_0})_{t \geq 0}$ is individually eventually strongly positive with respect to $u$.
\end{proposition}

\begin{proof}
    Let $0 < x \in \ell^2(\Z)$ and set $\alpha \coloneqq \dual{u, x} > 0$.
    By Proposition~\ref{proposition:reference-semigroup}, one has
    \begin{align*}
        \dual{Du, \ue^{tA_0}x} \to \dual{Du, u} \dual{u, x} = \alpha \dual{Du, u}
    \end{align*}
    as $t \to \infty$. Hence, there exists $t_0 > 0$ such that
    \begin{align*}
        \dual{D u, \ue^{tA_0} x} \geq \frac{\alpha \dual{Du, u}}{2}
    \end{align*}
    for all $t\geq t_0$. Duhamel's formula then implies
    \begin{align*}
        \ue^{tA_0}x = \ue^{-tD} x + \frac{1}{\dual{Du, u}} \int_0^t \ue^{-(t-s)D} Du \dual{Du, \ue^{sA_0}x}\, \ud s.
    \end{align*}
    For $t\geq t_0+1$, the inequality $\dual{De_k,e_k}\geq1$ yields
    \begin{align*}
        \dual{\ue^{tA_0}x,e_k}
        &\geq \frac{\dual{Du,e_k}}{\dual{Du,u}}
        \int_{t_0}^t
        \ue^{-\dual{De_k,e_k}(t-s)}
        \frac{\alpha\dual{Du,u}}{2}\,\ud s \\
        &= \frac{\alpha}{2}
        \left(
            1-\ue^{-\dual{De_k,e_k}(t-t_0)}
        \right)\dual{u,e_k} \\
        &\geq \frac{1-\ue^{-1}}{2}\alpha\dual{u,e_k}.
    \end{align*}
    Therefore, $\ue^{tA_0} x \succeq u$ for all $t \geq t_0 + 1$. Hence, $(\ue^{tA_0})_{t \geq 0}$ is individually eventually strongly positive with respect to $u$.
\end{proof}

Next, we prove the following technical lemma, which yields an estimate for the perturbation.

\begin{lemma} \label{lemma:kernel-growth}
    For each $n\in\N$, let $h_n\colon[0,\infty)\to\R$ be given by $h_n(t)\coloneqq\ue^{-(n-\delta_n)t}-\ue^{-nt}$. Then one has
    \begin{align*}
        \frac{\sup_{t \geq 0} h_n(t)}{u_n} \longrightarrow 0
    \end{align*}
    as $n \to \infty$.
\end{lemma}

\begin{proof}
    Clearly, one has
    \begin{align*}
        0\leq h_n(t) = \ue^{-(n-\delta_n)t} (1 - \ue^{-\delta_nt}) \leq \delta_n t \ue^{-(n-\delta_n)t}
    \end{align*}
    for all $t \geq 0$ and $n \in \N$. Since
    \begin{align*}
        \sup_{t \geq 0} t \ue^{-(n-\delta_n)t} = \frac{1}{\ue(n-\delta_n)},
    \end{align*}
    there exists a constant $C>0$ such that $\sup_{t\geq0}h_n(t) \leq C\ue^{-2n^2}$ for all $n\in\N$. Thus, it follows that
    \begin{align*}
        \frac{\sup_{t\geq0}h_n(t)}{u_n} \leq \frac{C \ue^{-2 n^2}}{\ue^{-3n^2/2}} = C \ue^{-n^2/2} \longrightarrow 0
    \end{align*}
    as $n \to \infty$.
\end{proof}

The preceding estimate shows that the perturbation is negligible relative to the limiting positive profile. We can now prove individual eventual strong positivity of the perturbed semigroup.

\begin{proposition} \label{proposition:individual}
The semigroup $(\ue^{tA})_{t \geq 0}$ is individually eventually strongly positive with respect to $u$.
\end{proposition}

\begin{proof}
    Fix $0 < x \in \ell^2(\Z)$. By Proposition~\ref{proposition:base-individually-eventually-positive}, there exist $t_0 \geq 0$ and $c > 0$ such that $\ue^{tA_0} x \geq cu$ for all $t\geq t_0$. Moreover, one has
    \begin{align*}
        (q_n\otimes q_n)x = \frac{x_n - x_{-n}}{2} (e_n - e_{-n})
    \end{align*}
    for all $n \in \N$ and hence
    \begin{align*}
        \abs{\dual{h_n(t)(q_n\otimes q_n)x,e_{\pm n}}} \leq \frac{h_n(t)}{\sqrt{2}}\norm{x}.
    \end{align*}
    By Lemma~\ref{lemma:kernel-growth}, there exists $N \in \N$ such that
    \begin{align*}
        \frac{\sup_{t \geq 0} h_n(t)}{\sqrt{2}}\norm{x}_2 \leq \frac{c}{2}u_n
    \end{align*}
    for all $n \geq N$. Thus, for all $t \geq t_0$ and $n \geq N$, one has
    \begin{align*}
    \dual{\ue^{tA}x, e_{\pm n}} &= \dual{\ue^{tA_0}x, e_{\pm n}} + \dual{h_n(t)(q_n \otimes q_n)x, e_{\pm n}} \\
    &\geq c u_{\pm n} -\frac{c}{2} u_{\pm n} = \frac{c}{2} u_{\pm n}.
    \end{align*}
    For each of the finitely many indices $1\leq n<N$, one has $h_n(t) \to 0$ as $t \to \infty$. Hence, after increasing $t_0$ if necessary, the same estimate holds for each $1\leq n<N$. Finally, $e_0 \in \ker B$, so
    \begin{align*}
    \dual{\ue^{tA}x, e_0} = \dual{\ue^{tA_0}x, e_0} \geq c\dual{u, e_0} \geq \frac{c}{2}\dual{u, e_0}
    \end{align*}
    for all $t \geq t_0$. Consequently, $\ue^{tA}x\succeq u$ for all $t\geq t_0$; that is, $(\ue^{tA})_{t\geq0}$ is individually eventually strongly positive with respect to $u$.
\end{proof}

It remains to show that the perturbed semigroup fails to be uniformly eventually positive. For this purpose, we estimate the symmetric contribution at suitable large times.

\begin{lemma} \label{lemma:reference-even-estimate}
    There exists $C > 0$ such that
    \begin{align*}
        \dual{\ue^{nA_0}|_{H_{\mathrm{even}}} p_n, p_n} \leq \ue^{-n^2} + C \ue^{-3n^2}
    \end{align*}
    for all sufficiently large $n \in \N$.
\end{lemma}

\begin{proof}
    We set $d_0\coloneqq1$ and $d_j\coloneqq j$ for $j\in\N$, and define
    \begin{align*}
        \pi_j \coloneqq
        \begin{cases}
            2 u_j^2, \quad &\text{if } j \in \N, \\
            u_0^2, \quad &\text{if } j = 0.
        \end{cases}
    \end{align*}
    and observe that the two identities
    \begin{align*}
        u = \sum_{j=0}^\infty \pi_j^{1/2} p_j, \qquad \norm{u}^2 = \sum_{j=0}^\infty \pi_j = 1.
    \end{align*}
    Moreover, we set $\kappa \coloneqq \dual{Du,u}$ and observe that
    \begin{align*}
        \kappa
        =
        \sum_{j=0}^\infty d_j\pi_j,
        \qquad
        \sum_{j=0}^\infty d_j^2\pi_j<\infty.
    \end{align*}
    Now consider the function
    \begin{align*}
        k(t) \coloneqq \frac{1}{\kappa} \dual{\ue^{-t D} Du, Du} = \frac{1}{\kappa} \sum_{j = 0}^\infty d_j^2 \pi_j \ue^{-d_jt}
    \end{align*}
    for $t \geq 0$. By Lemma~\ref{lemma:convolution}, the convolution resolvent
    \begin{align*}
        \rho \coloneqq \sum_{m = 1}^\infty k^{*m}
    \end{align*}
    is bounded, and thus, there exists $M > 0$ such that $\rho(t) \leq M$ for all $t \geq 0$.

    Since $A_0=-D+K$ for the bounded operator $K\coloneqq (Du\otimes Du)/\kappa$, Duhamel's formula for bounded perturbations \cite[Theorem~III.1.10]{Engel2000} yields
    \begin{align*}
        \ue^{tA_0}z
        =\ue^{-tD}z
        +\frac{1}{\kappa}\int_0^t
        \ue^{-(t-s)D}Du\dual{Du,\ue^{sA_0}z}\,\ud s
    \end{align*}
    for all $t\geq0$ and $z\in H$.

    For $n\in\N$, we define
    \begin{align*}
        y_n(t)
        \coloneqq
        \pi_n^{-1/2}
        \dual{\ue^{tA_0} p_n, Du}.
    \end{align*}
    All vectors occurring in this computation are real. Moreover, $Dp_n=np_n$ and $\dual{Du,p_n}=n\pi_n^{1/2}$. Hence, taking the inner product of the formula above with $Du$ and dividing by $\pi_n^{1/2}$ yields
    \begin{align*}
        y_n(t) = n \ue^{-nt} + \int_0^t k(t-s) y_n(s) \, \ud s.
    \end{align*}
    Proposition~\ref{proposition:volterra-resolvent} yields
    \begin{align*}
        y_n(t) = n \ue^{-nt} + \int_0^t \rho(t - s) n \ue^{-ns} \, \ud s
    \end{align*}
    for all $t \geq 0$. Hence, we obtain the estimate $y_n(t) \leq n \ue^{-n t} + M$ for all $t \geq 0$.

    Moreover,
    \begin{align*}
        \dual{\ue^{-(t-s)D}Du,p_n}
        =n\pi_n^{1/2}\ue^{-n(t-s)}.
    \end{align*}
    Hence, taking the inner product of the variation-of-constants formula with $p_n$ gives
    \begin{align*}
        \dual{\ue^{tA_0} p_n, p_n} = \ue^{-nt} + \frac{n \pi_n}{\kappa} \int_0^t \ue^{-n(t - s)} y_n(s) \,\ud s.
    \end{align*}
    Thus, we obtain
    \begin{align*}
        \dual{\ue^{nA_0} p_n, p_n} &\leq \ue^{-n^2} + \frac{n \pi_n}{\kappa} \int_0^n \ue^{-n(n - s)}(n \ue^{-ns} + M) \,\ud s \\
        &\leq \ue^{-n^2} + \frac{2 \ue^{-3n^2}}{\kappa}(n^3 \ue^{-n^2} + M).
    \end{align*}
    Therefore, there exists $C > 0$ such that
    \begin{align*}
        \dual{\ue^{nA_0} p_n, p_n} \leq \ue^{-n^2} + C \ue^{-3 n^2}
    \end{align*}
    for $n \in \N$ sufficiently large.
\end{proof}

We are now in a position to prove the main proposition of this section.

\begin{proposition} \label{proposition:failure-of-uniform-eventual-positivity}
    One has
    \begin{align*}
        \dual{\ue^{nA} e_n, e_{-n}} < 0
    \end{align*}
    for all sufficiently large $n \in \N$. In particular, $(\ue^{tA})_{t \geq 0}$ is not uniformly eventually positive.
\end{proposition}

\begin{proof}
    Note that one has
    \begin{align*}
        e_n = \frac{1}{\sqrt{2}}(p_n + q_n), \qquad e_{-n} =
        \frac{1}{\sqrt{2}} (p_n - q_n)
    \end{align*}
    for all $n \in \N$. Since $H_{\mathrm{even}}$ and $H_{\mathrm{odd}}$ reduce $A$ and $Aq_n = -(n - \delta_n)q_n$ for all $n \in \N$, it follows from Lemma~\ref{lemma:reference-even-estimate} that
    \begin{align*}
        \dual{\ue^{nA}e_n, e_{-n}} &= \frac 1 2 \big(\dual{\ue^{nA_0} p_n, p_n} - \ue^{-(n-\delta_n)n} \big) \\
        &\leq \frac1 2 \big(\ue^{-n^2} + C \ue^{-3n^2} - \ue^{-(n-\delta_n)n} \big) \\
        &\leq \frac1 2 \big(\ue^{-n^2} + C \ue^{-3n^2} - \ue^{-n^2} - n^2 \ue^{-3n^2} \big) \\
        &\leq \frac 1 2 (C - n^2) \ue^{-3n^2} < 0
    \end{align*}
    for sufficiently large $n \in \N$. Hence, $\ue^{nA}$ is not positive for arbitrarily large $n \in \N$, so $(\ue^{tA})_{t \geq 0}$ cannot be uniformly eventually positive.
\end{proof}

\begin{remark}
    The perturbation on the $n$-th antisymmetric mode satisfies
    \begin{align*}
        \sup_{t\geq0}
        \big(
            \ue^{-(n-\delta_n)t}-\ue^{-nt}
        \big)
        =
        \calO(\ue^{-2n^2}),
    \end{align*}
    whereas the corresponding coordinate of the Perron vector is
    \begin{align*}
        u_n=\ue^{-3n^2/2}.
    \end{align*}
    Thus, the perturbation is $o(u_n)$ and does not affect eventual strong positivity of a fixed positive orbit. At time $t=n$, however, the spectral shift contributes an excess of order $n^2\ue^{-3n^2}$ above the common leading term $\ue^{-n^2}$, whereas the corresponding symmetric correction is of order $\ue^{-3n^2}$. Hence, the off-diagonal entry is negative for arbitrarily large times.
\end{remark}

Finally, we prove the main theorem of the present paper.

\begin{proof}[Proof~of~Theorem~\ref{theorem:main}]
    Let $H=\ell^2(\Z)$ and let $u$ be the strictly positive unit vector defined in Section~\ref{section:construction-of-the-semigroup}. Proposition~\ref{proposition:properties-of-a}(ii) shows that $A$ has compact resolvent, while Proposition~\ref{proposition:semigroup-representation} shows that $A$ generates the self-adjoint semigroup $(\ue^{tA})_{t \geq 0}$. Since $A$ is real by Proposition~\ref{proposition:properties-of-a}(i), this semigroup is real. Proposition~\ref{proposition:individual} shows that it is individually eventually strongly positive with respect to $u$. Finally, Proposition~\ref{proposition:failure-of-uniform-eventual-positivity} shows that $\ue^{nA}$ is not positive for each sufficiently large $n \in \N$.

    Since $\ell^2(\Z)$ is an $L^2$-space over a $\sigma$-finite measure space, this disproves the equivalence in part~(a) of \cite[Open Problem~14.1.1]{Arora2026a}. Moreover, individual eventual strong positivity with respect to $u$ implies individual eventual positivity, whereas uniform eventual positivity with respect to $u\otimes u$ implies uniform eventual positivity. Thus, the question in \cite[Open Problem~14.1.1(b)]{Arora2026a} also has a negative answer.
\end{proof}

\appendix
\section{A scalar convolution estimate} \label{section:scalar-convolution-estimate}

We recall the terminology used in this section. Let $f,g \colon [0,\infty) \to \R$ be locally integrable. Their \emph{convolution} is defined by
\begin{align*}
    (f \ast g)(t) \coloneqq \int_0^t f(t - s) g(s) \, \ud s
\end{align*}
for all $t \geq 0$. We recursively define the \emph{convolution powers} of $f$ by $f^{\ast 1} \coloneqq f$ and
\begin{align*}
    f^{\ast (m+1)} \coloneqq f \ast f^{\ast m}
\end{align*}
for all $m \in \N$. The \emph{Laplace transform} of $f$ is defined by
\begin{align*}
    \widehat f(\lambda) \coloneqq \int_0^\infty \ue^{-\lambda t}f(t) \, \ud t
\end{align*}
whenever the integral exists.

A function $f \colon (0,\infty) \to [0,\infty)$ is called \emph{completely monotone} if $f \in C^\infty((0,\infty))$ and
\begin{align*}
    (-1)^m f^{(m)}(\lambda) \geq 0
\end{align*}
for all $\lambda > 0$ and $m \in \N_0$. A function $f \colon (0,\infty) \to [0,\infty)$ is called a \emph{Bernstein function} if $f \in C^\infty((0,\infty))$ and $f'$ is completely monotone. A Bernstein function is called a \emph{complete Bernstein function} if its L\'evy measure has a completely monotone density (see \cite[Definition~6.1]{Schilling2012a}).

A function $f \colon (0,\infty) \to [0,\infty)$ is called a \emph{Stieltjes function} if it admits a representation
\begin{align*}
    f(\lambda) = \frac{a}{\lambda} + b + \int_{(0,\infty)} \frac{1}{\lambda+s} \, \mu(\ud s)
\end{align*}
for some $a,b \geq 0$ and a positive measure $\mu$ on $(0,\infty)$ such that
\begin{align*}
    \int_{(0,\infty)} \frac{1}{1+s} \, \mu(\ud s) < \infty.
\end{align*}

We isolate the auxiliary convolution results needed in the proof of Proposition~\ref{proposition:failure-of-uniform-eventual-positivity}. They are formulated in a way that arises directly from Duhamel's formula for a rank-one perturbation of a diagonal operator.

\begin{lemma} \label{lemma:convolution}
    Let $(d_j)_{j \in \N_0}$ be a sequence in $[1,\infty)$ and let $(\pi_j)_{j \in \N_0}$ be a sequence in $(0,\infty)$ such that the following assertions hold:
    \begin{enumerate}[\upshape (a)]
        \item $\displaystyle \sum_{j \in \N_0} \pi_j = 1$.
        \item $\kappa \coloneqq \displaystyle \sum_{j \in \N_0}d_j \pi_j < \infty$.
        \item $\displaystyle \sum_{j \in \N_0} d_j^2 \pi_j < \infty$.
    \end{enumerate}
    Consider the function
    \begin{align*}
        k \colon [0, \infty) \to \R, \quad k(t)\coloneqq \frac{1}{\kappa} \sum_{j \in \N_0} d_j^2 \pi_j \ue^{- d_j t}.
    \end{align*}
    Then the convolution resolvent
    \begin{align*}
        \rho \colon [0, \infty) \to \R, \quad \rho\coloneqq\sum_{m=1}^{\infty}k^{\ast m}
    \end{align*}
    is well-defined and bounded on $[0,\infty)$.
\end{lemma}

\begin{proof}
    Clearly, the function $k$ is bounded and completely monotone. Moreover, one has
    \begin{align*}
        \int_0^\infty k(t) \, \ud t = \frac{1}{\kappa}\sum_{j \in \N_0}d_j \pi_j = 1,
    \end{align*}
    and thus $k \in L^1([0, \infty))$. The Laplace transform of $k$ is given by
    \begin{align*}
        \widehat{k}(\lambda) = \frac{1}{\kappa} \sum_{j \in \N_0} \frac{d_j^2\pi_j}{\lambda+d_j}
    \end{align*}
    for all $\lambda > 0$. Hence,
    \begin{align*}
        1 - \widehat{k}(\lambda) = \frac{1}{\kappa} \sum_{j \in \N_0} d_j \pi_j \bigg(1 - \frac{d_j}{\lambda + d_j} \bigg) = \frac{\lambda}{\kappa} \sum_{j \in \N_0} \frac{d_j \pi_j}{\lambda + d_j}.
    \end{align*}
    Set $K \coloneqq \norm{k}_\infty$. An induction shows that
    \begin{align*}
        0 \leq k^{\ast m}(t) \leq \frac{K^m t^{m-1}}{(m-1)!}
    \end{align*}
    for all $t \geq 0$ and $m \in \N$. Hence, the series
    \begin{align*}
        \rho = \sum_{m=1}^\infty k^{\ast m}
    \end{align*}
    converges locally uniformly on $[0,\infty)$ and defines a continuous function. Since $\widehat{k}(\lambda)<1$ for each $\lambda>0$, Tonelli's theorem yields
    \begin{align*}
        \widehat{\rho}(\lambda)
        = \sum_{m=1}^\infty \widehat{k}(\lambda)^m
        = \frac{\widehat{k}(\lambda)}{1-\widehat{k}(\lambda)}.
    \end{align*}
    The function
    \begin{align*}
        h(\lambda) \coloneqq \sum_{j \in \N_0} \frac{d_j\pi_j}{\lambda+d_j}
    \end{align*}
    is a Stieltjes function. Consequently, $\lambda h(\lambda)$ is a complete Bernstein function by \cite[Theorem~6.2]{Schilling2012a} and thus
    \begin{align*}
        \frac{1}{1-\widehat{k}(\lambda)} = \frac{\kappa}{\lambda h(\lambda)}
    \end{align*}
    is a Stieltjes function by \cite[Theorem~7.3]{Schilling2012a}. Since this Stieltjes function converges to $1$ as $\lambda \to \infty$, its constant term in the Stieltjes representation from \cite[Definition~2.1]{Schilling2012a} is equal to $1$. Hence,
    \begin{align*}
        \frac{\widehat{k}(\lambda)}{1-\widehat{k}(\lambda)} = \frac{1}{1-\widehat{k}(\lambda)} - 1
    \end{align*}
    is again a Stieltjes function, now with vanishing constant term. By \cite[Definition~2.1]{Schilling2012a} and the identity
    \begin{align*}
        \frac{1}{\lambda+t}=\int_0^\infty \ue^{-\lambda s}\ue^{-ts}\,\ud s,
    \end{align*}
    it is therefore the Laplace transform of a completely monotone function; see also \cite[Theorem~1.4]{Schilling2012a}. Uniqueness of the Laplace transform and continuity on $(0,\infty)$ show that this function coincides with $\rho$. Hence, $\rho$ is completely monotone and therefore decreasing. Finally, local uniform convergence of the series gives $\rho(0)=k(0)<\infty$, and thus $0 \leq \rho(t) \leq k(0)$ for all $t \geq 0$.
\end{proof}

\begin{proposition} \label{proposition:volterra-resolvent}
    Let $k$ and $\rho$ be as in Lemma~\ref{lemma:convolution}. For each $f \in L^1_{\mathrm{loc}}([0,\infty))$, the Volterra equation
    \begin{align*}
        y = f + k \ast y
    \end{align*}
    has a unique solution $y \in L^1_{\mathrm{loc}}([0,\infty))$, which is given by
    \begin{align*}
        y = f + \rho \ast f.
    \end{align*}
\end{proposition}

\begin{proof}
    The definition of $\rho$ yields $\rho = k + k \ast \rho$. Hence, $y \coloneqq f + \rho \ast f$ satisfies
    \begin{align*}
        f + k \ast y = f + (k + k \ast \rho) \ast f = f + \rho \ast f = y.
    \end{align*}
    To prove uniqueness, let $w \in L^1_{\mathrm{loc}}([0,\infty))$ satisfy $w = k \ast w$. Iteration gives $w = k^{\ast m} \ast w$ for each $m \in \N$. For $T > 0$, boundedness of $k$ implies
    \begin{align*}
        \norm{k^{\ast m}}_{L^1(0,T)} \leq \frac{\norm{k}_\infty^m T^m}{m!}.
    \end{align*}
    Young's inequality therefore gives
    \begin{align*}
        \norm{w}_{L^1(0,T)} \leq \frac{\norm{k}_\infty^m T^m}{m!}\norm{w}_{L^1(0,T)}.
    \end{align*}
    For sufficiently large $m$, the factor on the right is strictly less than $1$. Thus, $w = 0$ on $[0,T]$. Since $T > 0$ was arbitrary, uniqueness follows.
\end{proof}

\subsection*{AI disclosure statement}
During the preparation of this work, the author used an OpenCode harness, custom-written auxiliary skills and large language models from OpenAI's GPT-5.6 family for editorial assistance and to simplify an earlier version of the paper; this included assistance with the statement and proof of Lemma~\ref{lemma:convolution}. The author reviewed and verified all AI-assisted material and takes full responsibility for the content of the manuscript. All mathematical ideas, results and claims are due to the author.

\bibliographystyle{plain}
\bibliography{literature}

@book{Engel2000,
	author = {Engel, Klaus-Jochen and Nagel, Rainer},
	title = {One-Parameter Semigroups for Linear Evolution Equations},
	series = {Graduate Texts in Mathematics},
	volume = {194},
	publisher = {Springer},
	address = {New York},
	year = {2000},
}

@book{MeyerNieberg1991,
	author = {Meyer-Nieberg, Peter},
	title = {Banach Lattices},
	series = {Universitext},
	publisher = {Springer},
	address = {Berlin},
	year = {1991},
}

@book{Aliprantis2006,
  title={Positive operators},
  author={Aliprantis, Charalambos D and Burkinshaw, Owen},
  volume={119},
  year={2006},
  publisher={Springer Science \& Business Media}
}

@book{Zaanen2012,
  title={Introduction to operator theory in Riesz spaces},
  author={Zaanen, Adriaan C},
  year={2012},
  publisher={Springer Science \& Business Media}
}

@Book{Batkai2017,
  author     = {B\'{a}tkai, Andr\'{a}s and Kramar Fijav\v{z}, Marjeta and Rhandi, Abdelaziz},
  publisher  = {Birkh\"{a}user/Springer, Cham},
  title      = {Positive operator semigroups},
  year       = {2017},
  isbn       = {978-3-319-42811-6},
  note       = {From finite to infinite dimensions, With a foreword by Rainer Nagel and Ulf Schlotterbeck},
  series     = {Operator Theory: Advances and Applications},
  volume     = {257},
  doi        = {10.1007/978-3-319-42813-0},
  mrclass    = {47-02 (15-02 47B65 47D06)},
  mrreviewer = {Christoph Kriegler},
  pages      = {xvii+364},
  url        = {https://doi.org/10.1007/978-3-319-42813-0},
}

@Book{Schaefer1974,
  author     = {Schaefer, Helmut H.},
  publisher  = {Springer-Verlag, New York-Heidelberg},
  title      = {Banach lattices and positive operators},
  year       = {1974},
  series     = {Die Grundlehren der mathematischen Wissenschaften, Band 215},
  mrclass    = {46A40 (47B55 47D20)},
  mrreviewer = {A. C. Zaanen},
  pages      = {xi+376},
}

@Book{Aliprantis1999,
  author    = {Aliprantis, Charalambos D. and Border, Kim C.},
  publisher = {Springer-Verlag, Berlin},
  title     = {{I}nfinite-{D}imensional {A}nalysis},
  year      = {1999},
  edition   = {Second},
  isbn      = {3-540-65854-8},
  note      = {A hitchhiker's guide},
  doi       = {10.1007/978-3-662-03961-8},
  mrclass   = {46-01 (00A05 28-01 46N10 47-01 54-01)},
  pages     = {xx+672},
  url       = {https://doi.org/10.1007/978-3-662-03961-8},
}

@article{Daners2016b,
  title={Eventually positive semigroups of linear operators},
  author={Daners, Daniel and Gl{\"u}ck, Jochen and Kennedy, James B.},
  journal={Journal of Mathematical Analysis and Applications},
  volume={433},
  number={2},
  pages={1561--1593},
  year={2016},
  publisher={Elsevier}
}

@article{Daners2016a,
  title={Eventually and asymptotically positive semigroups on Banach lattices},
  author={Daners, Daniel and Gl{\"u}ck, Jochen and Kennedy, James B.},
  journal={Journal of Differential Equations},
  volume={261},
  number={5},
  pages={2607--2649},
  year={2016},
  publisher={Elsevier}
}

@article{Daners2017,
  title={The role of domination and smoothing conditions in the theory of eventually positive semigroups},
  author={Daners, Daniel and Gl{\"u}ck, Jochen},
  journal={Bulletin of the Australian Mathematical Society},
  volume={96},
  number={2},
  pages={286--298},
  year={2017},
  publisher={Cambridge University Press}
}

@article{Daners2018a,
  title={A Criterion for the Uniform Eventual Positivity of Operator Semigroups},
  author={Daners, Daniel and Gl{\"u}ck, Jochen},
  journal={Integral Equations and Operator Theory},
  volume={90},
  number={4},
  pages={46},
  year={2018},
  publisher={Springer}
}

@article{Daners2018b,
  title={Towards a perturbation theory for eventually positive semigroups},
  author={Daners, Daniel and Gl{\"u}ck, Jochen},
  journal={Journal of Operator Theory},
  volume={79},
  number={2},
  pages={345--372},
  year={2018},
  publisher={JSTOR}
}

@article{Daners2023,
  title={Local uniform convergence and eventual positivity of solutions to biharmonic heat equations},
  author={Daners, Daniel and Gl{\"u}ck, Jochen and Mui, Jonathan},
  journal={Differential and Integral Equations},
  volume={36},
  number={9/10},
  pages={727--756},
  year={2023},
  publisher={Khayyam Publishing, Inc. West Palm Beach, FL, USA}
}

@article{Glueck2022,
  title={Evolution equations with eventually positive solutions},
  author={Gl{\"u}ck, Jochen},
  journal={European Mathematical Society Magazine},
  volume={123},
  pages={4--11},
  year={2022}
}

@PhdThesis{GlueckDISS,
  author = {Gl{\"u}ck, Jochen},
  school = {Universit{\"a}t Ulm},
  title  = {Invariant sets and long time behaviour of operator semigroups},
  year   = {2017},
}

@Article{Arora2021a,
  author   = {Arora, Sahiba and Gl\"{u}ck, Jochen},
  journal  = {Semigroup Forum},
  title    = {Spectrum and convergence of eventually positive operator semigroups},
  year     = {2021},
  issn     = {0037-1912},
  number   = {3},
  pages    = {791--811},
  volume   = {103},
  doi      = {10.1007/s00233-021-10204-y},
  fjournal = {Semigroup Forum},
  mrclass  = {47D06 (47A10)},
  url      = {https://doi.org/10.1007/s00233-021-10204-y},
}

@PhdThesis{MuiDISS,
  author = {Mui, Jonathan},
  school = {University of Sydney},
  title  = {Eventual positivity and asymptotic behaviour for higher-order evolution equations},
  year   = {2023},
}

@article{Mui2023,
  title={Spectral properties of locally eventually positive operator semigroups},
  author={Mui, Jonathan},
  journal = {Semigroup Forum},
  volume={106},
  number={2},
  pages={460--480},
  year={2023},
}

@article{Arora2022a,
  author  = {Arora, Sahiba},
  title   = {Locally eventually positive operator semigroups},
  journal = {Journal of Operator Theory},
  volume  = {88},
  number  = {1},
  pages   = {203--242},
  year    = {2022}
}

@article{Arora2022b,
  author  = {Arora, Sahiba and Gl{\"u}ck, Jochen},
  title   = {An operator theoretic approach to uniform (anti-)maximum principles},
  journal = {Journal of Differential Equations},
  volume  = {310},
  pages   = {164--197},
  year    = {2022}
}

@article{Arora2022c,
  author  = {Arora, Sahiba and Gl{\"u}ck, Jochen},
  title   = {Stability of (eventually) positive semigroups on spaces of continuous functions},
  journal = {Comptes Rendus Math{\'e}matique},
  volume  = {360},
  pages   = {771--775},
  year    = {2022}
}

@phdthesis{Arora2023a,
  author = {Arora, Sahiba},
  title  = {Long-term behaviour of operator semigroups and (anti-)maximum principles},
  school = {TU Dresden},
  year   = {2023}
}

@article{Arora2023b,
  author  = {Arora, Sahiba and Gl{\"u}ck, Jochen},
  title   = {A characterization of the individual maximum and anti-maximum principle},
  journal = {Mathematische Zeitschrift},
  volume  = {305},
  number  = {2},
  pages   = {Paper No. 24},
  note    = {17 pages},
  year    = {2023}
}

@incollection{Arora2023c,
  author    = {Arora, Sahiba and Gl{\"u}ck, Jochen},
  title     = {Criteria for eventual domination of operator semigroups and resolvents},
  booktitle = {Operators, Semigroups, Algebras and Function Theory},
  pages     = {1--26},
  publisher = {Birkh{\"a}user},
  address   = {Cham},
  year      = {2023}
}

@article{Arora2024a,
  author  = {Arora, Sahiba and Gl{\"u}ck, Jochen},
  title   = {Irreducibility of eventually positive semigroups},
  journal = {Studia Mathematica},
  volume  = {276},
  number  = {2},
  pages   = {99--129},
  year    = {2024}
}

@article{Arora2025a,
  author  = {Arora, Sahiba},
  title   = {Eventually positive semigroups: spectral and asymptotic analysis},
  journal = {Semigroup Forum},
  volume  = {110},
  number  = {2},
  pages   = {263--295},
  year    = {2025}
}

@misc{Arora2025b,
  author        = {Arora, Sahiba and Mui, Jonathan},
  title         = {Smoothing of operator semigroups under relatively bounded perturbations},
  year          = {2025},
  eprint        = {2501.18556},
  archivePrefix = {arXiv},
  primaryClass  = {math.FA}
}

@misc{Arora2026a,
  author = {Arora, Sahiba and Gl{\"u}ck, Jochen and Mui, Jonathan},
  title  = {Eventual Positivity},
  year   = {2026},
  url    = {https://fan.uni-wuppertal.de/fileadmin/mathe/reine_mathematik/funktionalanalysis/glueck/Sonstiges/2025-10_ISEM29/lecture_notes/ISem_29__Up_to_Chapter_14.pdf}
}

@Book{Engel2006,
  author     = {Engel, Klaus-Jochen and Nagel, Rainer},
  publisher  = {Springer, New York},
  title      = {A short course on operator semigroups},
  year       = {2006},
  isbn       = {978-0387-31341-2},
  series     = {Universitext},
  mrclass    = {47-01 (47D03 47D06)},
  mrnumber   = {2229872},
  mrreviewer = {Jacek Banasiak},
  pages      = {x+247},
}

@Book{Hille1957,
  author     = {Hille, Einar and Phillips, Ralph S.},
  publisher  = {American Mathematical Society, Providence, RI},
  title      = {Functional analysis and semi-groups},
  year       = {1957},
  note       = {rev. ed},
  series     = {American Mathematical Society Colloquium Publications, Vol. 31},
  mrclass    = {46.2X},
  mrnumber   = {89373},
  mrreviewer = {M. H. Stone},
  pages      = {xii+808},
}

@Book{Tanabe1979,
  author     = {Tanabe, Hiroki},
  publisher  = {Pitman (Advanced Publishing Program), Boston, Mass.-London},
  title      = {Equations of evolution},
  year       = {1979},
  isbn       = {0-273-01137-5},
  note       = {Translated from the Japanese by N. Mugibayashi and H. Haneda},
  series     = {Monographs and Studies in Mathematics},
  volume     = {6},
  mrclass    = {47D05 (34G10 35K22 47H05)},
  mrnumber   = {533824},
  mrreviewer = {A. Pazy},
  pages      = {xii+260},
}

@Book{Davies1980,
  author     = {Davies, Edward B.},
  publisher  = {Academic Press, Inc. [Harcourt Brace Jovanovich, Publishers], London-New York},
  title      = {One-parameter semigroups},
  year       = {1980},
  isbn       = {0-12-206280-9},
  series     = {London Mathematical Society Monographs},
  volume     = {15},
  mrclass    = {47D05},
  mrnumber   = {591851},
  mrreviewer = {S.\ Kurepa},
  pages      = {viii+230},
}

@Book{Pazy1983,
  author     = {Pazy, Amnon},
  publisher  = {Springer-Verlag, New York},
  title      = {Semigroups of linear operators and applications to partial differential equations},
  year       = {1983},
  isbn       = {0-387-90845-5},
  series     = {Applied Mathematical Sciences},
  volume     = {44},
  doi        = {10.1007/978-1-4612-5561-1},
  mrclass    = {47D05 (34Gxx 35Fxx 35Gxx 47H20)},
  mrnumber   = {710486},
  mrreviewer = {H. O. Fattorini},
  pages      = {viii+279},
  url        = {https://doi.org/10.1007/978-1-4612-5561-1},
}

@Book{Goldstein1985,
  author     = {Goldstein, Jerome A.},
  publisher  = {The Clarendon Press, Oxford University Press, New York},
  title      = {Semigroups of linear operators and applications},
  year       = {1985},
  isbn       = {0-19-503540-2},
  series     = {Oxford Mathematical Monographs},
  mrclass    = {47D05 (34K30 35R20)},
  mrnumber   = {790497},
  mrreviewer = {H. O. Fattorini},
  pages      = {x+245},
}

@Book{NagelEd,
  editor     = {Nagel, Rainer},
  publisher  = {Springer-Verlag, Berlin},
  title      = {One-parameter semigroups of positive operators},
  year       = {1986},
  isbn       = {3-540-16454-5},
  series     = {Lecture Notes in Mathematics},
  volume     = {1184},
  doi        = {10.1007/BFb0074922},
  mrclass    = {47D05 (46L55 47B55)},
  mrreviewer = {J. A. van Casteren},
  pages      = {x+460},
  shorthand  = {Nag86},
  url        = {https://doi.org/10.1007/BFb0074922},
}

@Book{Lunardi1995,
  author     = {Lunardi, Alessandra},
  publisher  = {Birkh\"{a}user Verlag, Basel},
  title      = {Analytic semigroups and optimal regularity in parabolic problems},
  year       = {1995},
  isbn       = {3-7643-5172-1},
  series     = {Progress in Nonlinear Differential Equations and their Applications},
  volume     = {16},
  doi        = {10.1007/978-3-0348-9234-6},
  mrclass    = {47D06 (34G20 35Kxx 46M35 46N20 47-02 47N20 58D25)},
  mrnumber   = {1329547},
  mrreviewer = {Paolo\ Acquistapace},
  pages      = {xviii+424},
  url        = {https://doi.org/10.1007/978-3-0348-9234-6},
}

@Book{Bobrowski2016,
  author     = {Bobrowski, Adam},
  publisher  = {Cambridge University Press, Cambridge},
  title      = {Convergence of one-parameter operator semigroups},
  year       = {2016},
  isbn       = {978-1-107-13743-1},
  note       = {In models of mathematical biology and elsewhere},
  series     = {New Mathematical Monographs},
  volume     = {30},
  doi        = {10.1017/CBO9781316480663},
  mrclass    = {47-02 (47D06 92D10 92D25)},
  mrnumber   = {3526064},
  mrreviewer = {Przemo\ Kranz},
  pages      = {xiv+438},
  url        = {https://doi.org/10.1017/CBO9781316480663},
}

@book{Kato1995a,
  author    = {Kato, Tosio},
  title     = {Perturbation Theory for Linear Operators},
  edition   = {2},
  publisher = {Springer},
  address   = {Berlin},
  year      = {1995},
  series    = {Classics in Mathematics}
}

@book{Schilling2012a,
  author    = {Schilling, Ren{\'e} L. and Song, Renming and Vondra{\v{c}}ek, Zoran},
  title     = {Bernstein Functions: Theory and Applications},
  edition   = {2},
  series    = {De Gruyter Studies in Mathematics},
  volume    = {37},
  publisher = {De Gruyter},
  address   = {Berlin},
  year      = {2012}
}

\end{document}